\documentclass[a4paper,12pt,leqno]{amsart}
\usepackage{graphicx}
\usepackage{pinlabel} %for algebraic and geometric topology
\usepackage{amsmath}
\usepackage{amsfonts}
\usepackage{amssymb}
\usepackage{mathtools}
\usepackage[all,cmtip]{xy}
\usepackage{amsthm}
\usepackage{tikz-cd}
\usepackage{comment}
\usepackage{enumerate}
\usepackage{url}
\usepackage{epsfig}
\usepackage{hyperref}
\usepackage[utf8]{inputenc}
\usepackage[capitalize]{cleveref}
\usepackage{geometry}

\usepackage{subcaption}
\usepackage{mathrsfs}
\usepackage{xcolor}

\makeatletter
\providecommand\@dotsep{5}
\def\listtodoname{List of Todos}
\def\listoftodos{\@starttoc{tdo}\listtodoname}
\makeatother

\newtheorem{theorem}{Theorem}[section]
\newtheorem{proposition}[theorem]{Proposition}
\newtheorem{corollary}[theorem]{Corollary}
\newtheorem{lemma}[theorem]{Lemma}

\newcommand{\mycomment}[1]{}

  \theoremstyle{definition}
\newtheorem{definition}[theorem]{Definition}
\newtheorem{example}[theorem]{Example}
\newtheorem{remark}{Remark}
\newtheorem{question}[theorem]{Question}

\newcommand{\calA}{{\mathcal A}}

\newcommand{\calF}{{\mathcal F}}

\newcommand{\calY}{\mathcal Y}
\newcommand{\asA}{{\mathbf A}} %arc=a system=s A
\newcommand{\asB}{{\mathbf B}} %arc=a system=s B
\newcommand{\T}{\mathscr T}

\newcommand{\nbeq}{\begin{equation}}
\newcommand{\neeq}{\end{equation}}
\newcommand{\beq}{\begin{equation*}}
\newcommand{\eeq}{\end{equation*}}

\DeclareMathOperator{\vcd}{vcd}

\DeclareMathOperator{\gdfin}{\underline{gd}}

\DeclareMathOperator{\Mod}{Mod}
\DeclareMathOperator{\PMod}{PMod}

\newcommand{\modgs}{\Mod_{g}^{s}}
\newcommand{\pmodgs}{\PMod_{g}^{s}}

\newcommand{\Tgs}{\T_{g}^{s}}

\DeclareMathOperator{\rank}{rank}

\usepackage{verbatim}

\begin{document}

\title[]{On the dimension of Harer's spine\\ for the decorated Teichmüller space}
%\title[]{On the classifying space of the family of virtually cyclic subgroups for normally poly-free groups}

\author[N. Colin]{Nestor Colin}
\author[R. Jiménez Rolland]{Rita Jiménez Rolland}
\author[P. L. León Álvarez ]{Porfirio L. León Álvarez}
\address{Instituto de Matemáticas, Universidad Nacional Autónoma de México. Oaxaca de Juárez, Oaxaca, México 68000}
\email{rita@im.unam.mx}
\email{ncolin@im.unam.mx}
\email{porfirio.leon@im.unam.mx}

\author[L. J. Sánchez Saldaña]{Luis Jorge S\'anchez Salda\~na}
\address{Departamento de Matemáticas, Facultad de Ciencias, Universidad Nacional Autónoma de México}
\email{luisjorge@ciencias.unam.mx}


%\address{Instituto de Matemáticas, Universidad Nacional Autónoma de México. Oaxaca de Juárez, Oaxaca, México 68000}
%\email{porfirio.leon@im.unam.mx}

%    General info

\date{}

\keywords{Mapping class groups, Teichm\"uller space, spines, classifying spaces for proper actions}

\subjclass{57K20, 55R35, 20J05, 57M07, 57M50, 57M60}

\begin{abstract}
  In \cite{Ha86} Harer  explicitly constructed a spine for the decorated Teichm\"uller space of orientable surfaces with at least one puncture and negative Euler characteristic. In this paper we point out  some instances where his computation of the dimension of this spine is off by $1$ and give the correct dimension. %In this paper we point out that, in some {\color{red} instances/cases ?}, his computation of the dimension is off by $1$ and give the correct dimension.% However, his computation of the dimension is off by $1$ in some cases of the construction. In this paper we obtain the correct dimension of the spine and point out some instances where it has been misquoted. %
\end{abstract}
\maketitle

%\tableofcontents

\section{Introduction}

%{\color{red} In this note we will consider mainly the case with empty boundary??}

%Let $F_g^{s}$ be an orientable closed connected surface of genus $g$ and  $s$  punctures $\{p_1, \ldots, p_s\}$.
 Let \( F \) be an orientable closed connected surface of genus \( g \) and consider \( \{p_1, \ldots, p_s\} \)  a collection of distinguished points in \( F\).   The {\it mapping class group} $\modgs$ is the group of isotopy classes of all orientation-preserving  diffeomorphisms  $f:F_g^s\rightarrow F_g^s$, where $F_g^{s}:=F-\{p_1, \ldots, p_s\}$. The group $\modgs$ permutes the  set $\{p_1, \ldots, p_s\}$ %punctures 
and the kernel of such action  is the {\it pure mapping class group} $\pmodgs$. These groups are related by the following short exact sequence
$$
    1\rightarrow \pmodgs\rightarrow \modgs\rightarrow \frak{S}_s\rightarrow 1,
$$
where $\frak{S}_s$ denotes the symmetric group on $s$ letters.  The  groups $\modgs$ and $\pmodgs$ act on the \textit{Teichmüller space} $\Tgs$ of conformal equivalence classes of marked Riemann surfaces of genus $g$ with $s$ distinguished points.

 In the influential paper \cite{Ha86}, J. L. Harer established several cohomology properties of $\PMod_{g}^s$, including the computation of its virtual cohomological dimension: 
\[
\vcd(\pmodgs) =
\begin{cases}
s-3 & \text{if } g=0 \text{ and } s\geq 3 \\
1 & \text{if } g=1 \text{ and } s= 0 \\
4g-5 & \text{if } g\geq 2 \text{ and } s=0\\
4g-4+s & \text{if } g\geq 1 \text{ and } s\geq 1.
\end{cases}
\]
 In particular, for $s\geq 1$, $2g+s>2$,  and   $\Delta =\{p_1, \ldots, p_m\}$ with $1\leq m\leq s$,  Harer explicitly described  a cell complex $\mathcal{Y}(=\mathcal{Y}_g^{s,m})$  inside the {\it decorated Teichm\"uller space} $\Tgs(\Delta)$ %\approx\Tgs\times(\Delta^{m-1})^{\circ}$  which is homeomorphic to the product of $\Tgs$ and a open simplex of dimension $m-1$, 
 onto which  $\Tgs(\Delta)$ may be $\PMod_{g}^s$-equivariantly retracted.  
In \cite[Theorem 2.1]{Ha86} Harer states  that the dimension of $\mathcal{Y}$ is $4g-5+s+m$. However, while revising the details of his construction we observed that this computation  is off by $1$ when the parameter $m<s$; see for instance \cref{EXA} below. In this paper we compute the correct dimension of $\mathcal{Y}$. 

In what follows we refer to $\mathcal{Y}$ as  {\it Harer's spine of the decorated Teichm\"uller space}. 

\begin{theorem}[Dimension of Harer's spine]\label{thm:main}
Let $s\geq 1$, \  $2g+s>2$, and $1\leq m\leq s$. The dimension of the spine $\calY(=\mathcal{Y}_g^{s,m})$ is given by 

    \[
\dim(\calY) = \begin{cases}
    4g-4+s+m & \text{ if } m<s\\
    4g-5+s+m & \text{ if } m=s.
\end{cases} 
\]
\end{theorem}

%\begin{itemize}

%\item For $g=0$, $s\geq 3$ and $m=1$, $\mathcal{Y}$ is a spine for $\mathcal{T}_0^s$ of minimal dimension $\dim\mathcal{Y}=s-3=\vcd(\PMod_0^s)$.

%\item For $g\geq 1$,  $s=m=1$, $\dim\mathcal{Y}=4g-4+s=\vcd(\PMod_g^1)$, then $\mathcal{Y}$ is a spine for $\mathcal{T}_g^1$ of  minimal dimension.

%\item However, for $g\geq 1$, $s\geq 2$ and $m=1$, $\dim\mathcal{Y}=4g-4+s+1=\vcd(\PMod_g^s)+1$ a and $\mathcal{Y}$ is a spine for $\mathcal{T}_g^s$ which is not minimal dimension.
%\end{itemize}

 %En \cite{Ha86} Harer calculó la vcd para Mapping Class Groups de superficies con característica de Euler negativa... En ese mismo artículo Harer exhibió una triangulación PMod equivriante (no Mod equivariante salvo en casos muy particulares) ideal de el espacio de Teichmüller decorado junto con una espina natural asociada. En este artículo puntualizamos un pequeño error en el cálculo de la dimensión de dicha espina... y calcular muy explícitamente la dimensión de dicha espina.

  In the literature, Harer's spine $\mathcal{Y}$ (when $m=1$) has been cited to justify the existence of a model for $\underline E \modgs$, the {\it classifying space for proper actions of $\modgs$}, of minimal dimension. This use of Harer's spine is not correct when $s\geq 2$, and part of our motivation to write this paper is to clarify it. %We want to point out that this use of Harer's spine is not correct when $s\geq 2$, this was the original motivation to write this note. 
  The misquotation comes from two directions:

  \begin{itemize}
      \item[a)] {\it The complex $\calY(=\calY_g^{s,1})$ is not a model for $\underline E \modgs$ when $s\geq2$.} For $2g+s>2$, the Teichmüller space $\Tgs$ is known to be a model for $\underline E \modgs$; see for instance \cite[Proposition 2.3]{WolpertJi} and \cite[Section 4.10]{Lu05}. When $m=1$, Harer's spine $\calY$ is a $\pmodgs$-equivariant deformation retraction of $\Tgs(\Delta)=\Tgs$ and  hence it is a model for $\underline E\pmodgs$. On the other hand, from the definition of Harer's spine, for $m<s$ the complex $\calY$ only admits a $\pmodgs$-action; see \cref{sec:spine} for details. Therefore, despite $\Tgs$  admitting a $\modgs$-action, the spine $\calY$ cannot be a $\modgs$-equivariant retraction of $\Tgs$, and it is not a model for $\underline E \modgs$ when $s\geq 2$. This misquotation appears in \cite[Introduction]{WolpertJi}, \cite[Proof of Theorem 1.1]{AMP14}, \cite[Introduction]{Ji14}, \cite[Subsection 3.2]{Ramon:Juan}, and \cite[Proof of Theorem 1.5]{JRLASS24}. \\

      \item[b)] {\it The dimension of $\calY(=\calY_g^{s,1})$ is not always equal to $\vcd(\pmodgs)$.} For the case $g=0$, $s\geq 3$ and $m=1$, and   the case $g\geq 1$  and  $s=m=1$, we have that $\dim\mathcal{Y}=\vcd(\PMod_g^s)$.  However,  for $g\geq 1$ and  $s\geq 2$, the computation in \cref{thm:main} shows that $$\dim\mathcal{Y}=4g-4+s+1=\vcd(\PMod_g^s)+1$$ and hence $\calY$ is not a model of minimal dimension in this case.  This misquotation appears in  \cite[Section 1]{BV06}, \cite[Section 2.2]{Gui2010} \cite[Introduction]{AMP14}, \cite[Introduction]{Ji14}, \cite[Section 10]{Hensel14},  \cite[Subsection 3.2]{Ramon:Juan}, and \cite[Proof of Theorem
1.5]{JRLASS24}. We indicate in \cref{Propgd} how the dimension of  Harer's spine  when $s=m=1$, and the Birman exact sequence can be used to show the existence of a model for $\underline E \pmodgs$ of minimal dimension when $s\geq 2$.  \\
  \end{itemize}

To the best of our knowledge it is still unknown whether there exists a model for $\underline E\modgs$ of dimension $\vcd(\modgs)$ for $s\geq 2$.  Although \cite[Corollary 1.3]{AMP14} claims the existence of such models, their proof \cite[Remark 4.5]{AMP14} relies on an inductive argument using the Birman short exact sequence and it only applies for pure mapping class groups. We want to point out that \cite[Corollary 1.3]{AMP14} is cited for instance in \cite[Theorem 4.1, Proof of Theorem 1.5]{AJPTN18}, \cite[Introduction, Theorem 5.1]{JPT16}, and \cite[Proof of Proposition 5.3]{Nucinkis:Petrosyan}.  

When $s=m$, it follows from Harer's construction that the retraction of $\Tgs(\Delta)$ onto Harer's spine $\calY(=\calY_g^{s,s})$ is actually $\modgs$-equivariant. Such spine $\calY$ is of dimension  $4g-5+2s$ and it can be shown that it gives a model for $\underline E \modgs$. Since $\Tgs(\Delta)$ is topologized as the product of $\Tgs$ with an open $(s-1)$-simplex, and $\modgs$ acts diagonally, there is  a natural $\modgs$-projection $\pi:\Tgs(\Delta)\rightarrow \Tgs$.   Hence, the image  $\pi(\calY)$  is a $\modgs$-equivariant deformation retract of $\Tgs$ and it is a model for $\underline E \modgs$  of dimension at most $4g-5+2s$. 
\begin{question} Is $\pi(\calY)$ a model for $\underline E \modgs$ of minimal dimension  when $s\geq 1$ and $2g+s>2$?
\end{question}
%We would be interested to know whether $\pi(\calY)$ is of minimal dimension  when $s\geq 1$ and $2g+s>2$.  

In a forthcoming paper \cite{Los4} we use Harer's spine to construct a spine for the Teichm\"uller space of non-orientable punctured surfaces $N_g^s$ (with $s\geq 1$) with negative Euler characteristic. When $s=1$ this spine gives us a model of minimal dimension for $\underline{E}\Mod(N_g^1)$. \\
 
{\noindent\bf Some comments on spines for Teichm\"uller space $\T_{g}$}. A related question is whether Teichmüller space $\T_{g}$ admits a $\Mod_g$-equivariant deformation retraction onto a cocompact spine whose dimension is equal to $\vcd(\Mod_g)$, see \cite[Question 1]{BV06}.  Since Harer's construction needs at least one marked point in the surface, it cannot be used to address this question.

%\textcolor{blue}{ {\color{red} In the case of a closed surface of genus $g\geq 3$ it is still open the question of whether the associated Teichmüller space $\T_{g}$ admits a spine of dimension $\vcd(\Mod_g)$.??  Harer's construction needs at least one marked point in the surface.} 
For $g\geq 2$, W. Thurston proposed a candidate for a spine for $\T_g$ in a hard to find preprint \cite{Th85}; see also \cite[Remark 4.4]{Ji14}. M. Fortier Bourque proved in \cite[Theorem 1.1]{FB24} that in general this spine is not of minimal dimension. In a recent preprint \cite[Theorem 1]{Irmer2022} I. Irmer claims the existence of  a $\Mod_g$-equivariant deformation retraction of the Thurston spine onto a CW-complex of dimension equal to $\vcd(\Mod_g)$. 

 On the other hand, for any genus $g\geq 1$, S.A. Broughton  proved in \cite[Theorem 2.7]{Bro90}   that  $\T_g$ contains a $\Mod_g$-subspace which is a strong $\Mod_g$-deformation retract, and which is a co-compact model for $\underline E \Mod_g$; see  also \cite[Section 3]{Gui2010}.  L. Ji  constructed in \cite{Ji14} a couple of spines for $\T_{g}$ that give cocompact models for $\underline E \Mod_g$.  Ji proves that one his constructions gives a spine of codimension at least $2$ in $\T_{g}$ \cite[Proposition 4.3]{Ji14} and for genus $g=2$ it is of minimal dimension.

Furthermore, in \cite[Theorem 1.1]{AMP14} J. Aramayona and C. Martínez-Pérez proved that there exists a cocompact model for $\underline E \mathrm{Mod}_g$ of dimension $\vcd(\mathrm{Mod}_g)$,  for every $g\geq 0$. Their proof is not constructive and does not give a spine for $\T_{g}$ when $g\geq 2$. % We want to point out that their argument has a flaw for $g=2$ since they use Harer's spine for the full mapping class group of $\Mod_0^6$, but as mentioned early Ji's spine can be used as a replacement of this argument.
For genus $g=2$, their argument has a flaw: it relies on the fact that $\Mod_2$ is a central extension of $\Mod_0^6$ by $\mathbb{Z}_2$, and it assumes incorrectly that Harer's spine gives a co-compact model for $\underline E \mathrm{Mod}_0^6$; see item a) above. % to justify the existence of a model for $\underline E \mathrm{Mod}_2$ of minimal dimension. 
Instead, Ji's spine of minimal dimension for genus $g=2$ can be cited.% to complete their proof. 

%(hence a model for   It relies on the fact that $\Mod_2$ is a central extension of $\Mod_0^6$ by $\mathbb{Z}_2$. They then claim that given that Harer's spine gives a co-compact model for $\underline E \mathrm{Mod}_0^6$ of dimension $\vcd(\Mod_0^6)=\Mod_2$, it gives a model for  $\underline E \mathrm{Mod}_2$ of minimal dimension. However, as we already pointed out Harer's spine is not a model for $\underline E \mathrm{Mod}_0^6$.   
\bigskip
%{\color{red} $\bullet$ Ideal triangulation of the Teichmüller space for surfaces with at least one puncture}

\noindent{\bf Organization of the paper.}  In \cref{sec:arcs} we recall the definition of the arc complex $\mathcal{A}=\mathcal{A}(\Delta)$ considered by Harer in \cite[Section 1]{Ha86} and compute its dimension. This arc complex defines an ideal triangulation $\mathcal{A}(\Delta)-\mathcal{A}_{\infty}(\Delta)$ of the decorated Teichm\"uller space $\Tgs(\Delta)$ which Harer used to construct his spine in \cite[Section 2]{Ha86}.  In \cref{sec:spine} we review the definition of Harer's spine and compute its dimension proving \cref{thm:main}.\bigskip 

\noindent{\bf Acknowledgments.}  We thank Javier Aramayona, Mladen Bestvina, Maxime Fortier Bourque, John Harer, Conchita Mart\'inez-P\'erez, and Andrew Putman  for useful communication. The first author was funded by CONAHCyT through the program \textit{Estancias Posdoctorales por México.} The third author's work was supported by UNAM \textit{Posdoctoral Program (POSDOC)}. All authors are grateful for the financial support of DGAPA-UNAM grant PAPIIT IA106923.

\section{Arc systems and Harer's complex of arcs}\label{sec:arcs}

 We recall the definition of Harer's complex of arcs $\mathcal{A}=\calA(\Delta)$ from \cite[Section1]{Ha86}. %and  give the details of the computation of its dimension in \cref{thm:arc:dim} }
For the sake of completeness we include a full-detailed computation of its dimension in \cref{thm:arc:dim}.% by means of a computation the dimension of maximal simplices.}
%\subsection{Arc systems and the Harer complex (mostly definitions)}
%\begin{definition}
%A simplicial subcomplex $L$ of a simplicial complex $K$ is called full if every simplex in $K$ whose vertices all belong to $L$ is itself in $L$.
%\end{definition}

Let \( F \) be a closed, orientable surface of genus \( g \) %with \( r \) boundary components 
and let \( \{p_1, \ldots, p_s\} \) be a collection of distinguished points in \( F\). % - \partial F \). 
If \( \Delta = \{p_1, \ldots, p_m\} \), \( m \leq s \), we write \( P = \{p_{m+1}, \ldots, p_s\} \) and \( F_0 = F - P \) so that the surface \( F_0 \) has \( m \) distinguished points and \( s-m \) punctures. In what follows, the parameters $r$ and $n$ appearing in \cite[Section 1]{Ha86} are assumed to be zero.

%Next let \( \Delta_2 = \{q_1, \ldots, q_n\} \) be a set of points in \( \partial F \), where we assume for technical reasons that each component of \( \partial F \) contains at least one point of \( \Delta_2 \). Put \( \Delta = \Delta_1 \cup \Delta_2 \) and suppose from here on that \( \Delta \) is nonempty.

A properly imbedded path in \( F_0 \) between two points of \( \Delta \) %or an imbedded loop in \( F_0 \) based at a point of \( \Delta \) and meeting \( \partial F \) only at this point if at all 
will be called a \( \Delta \)-arc. The isotopy class in \( F_0 \) (rel \( \Delta \)) \( [\alpha_0, \ldots, \alpha_k] \) of a family of \( \Delta \)-arcs will be called a {\it rank-\( k \) arc system} if:

\begin{enumerate}
\item \( \alpha_i \cap \alpha_j \subset \Delta \) for distinct \( i \) and \( j \), and
 \item for each connected component \( B \) of the surface obtained by \textit{ splitting \( F_0 \) along \( \alpha_0, \ldots, \alpha_k \)}, the Euler characteristic of the double of \( B \) along \( \partial B - \Delta \) is negative. 
\end{enumerate}

The condition (2) ensures that $\Delta$-arcs are not null-homotopic (rel $\Delta$), and no two distinct $\Delta$-arcs are homotopic (rel $\Delta$).

\begin{remark}
    Let us give a detailed explanation of \textit{the connected components} appearing in condition (2). Consider the surface $F_0$ and a collection of $\Delta$-arcs \( \alpha_0, \ldots, \alpha_k \) such that  condition (1) holds. Let $B^o$ be one of the connected components of the open surface $F_0-\cup_{i=0}^k \alpha_i$. Hence $B^o$ is the interior of a surface $B$ with non-empty boundary, and we have an \textit{attaching map} $\phi_B\colon \partial B\to \cup_{i=0}^k \alpha_i$. Since $\cup_{i=0}^k \alpha_i$ is canonically a 1-dimensional CW-complex (with 0-skeleton contained in $\Delta$), we can pull-back that cellular structure to $\partial B$. By an abuse of notation we call $\Delta$ the $0$-skeleton of $\partial B$.  We will be considering each $B$ as a polygon possibly with punctures and every edge will be labeled with the corresponding $\Delta$-arc in the the system. 
\end{remark}

\begin{remark}\label{arc:system:deter:CW:structure}
    Assume $s=m$, that is we have no punctures and $F_0=F$. If each connected component of $F-\cup_{i=1}^k \alpha_i$ is homeomorphic to a disk, then the $\Delta$-arcs determine a CW-structure on $F$.  
    %\textcolor{red}{Lo de etiquetar las 1-célula de la frontera de las piezas lo vamos a hacer incluso cuando tengamos punturas. No estoy seguro dónde queremos poner esta frase, quizá en el remark anterior o hacer un nuevo remark?} 
\end{remark}

\begin{definition}[Harer's complex of arcs] Let \(\calA=\calA(\Delta)\) be the  simplicial complex that has  a \(k\)-simplex \(\langle \alpha_0, \ldots, \alpha_k\rangle\) for each rank-\(k\) arc system in \(F_0\) and such that \(\langle\beta_0, \ldots, \beta_l\rangle\) is  a face of \(\langle\alpha_0, \ldots, \alpha_k\rangle\) if \(\{[\beta_0], \ldots, [\beta_l]\}\) is contained in \(\{[\alpha_0], \ldots, [\alpha_k]\}\). 
    
\end{definition}

%\textcolor{red}{LJ: Harer proves in \cite[Theorem 1.3]{Ha86} that there is a $\pmodgs$-equivariant homeomorphism $\Tgs(\Delta)\to \calA(\Delta)-\calA_\infty(\Delta)$ where $\calA_\infty(\Delta)$ is a codimension 2 subcomplex of $\calA(\Delta)$. As a consequence the dimensions of $\calA(\Delta)$ and $\Tgs(\Delta)$ are equal. Anyway, for the sake of completeness we include a full-detailed computation of the dimension of $\calA(\Delta)$ by means of a computation the dimension of maximal simplices.}

\begin{theorem}\label{thm:arc:dim}
The dimension of \(\calA(\Delta)\) is $6g-7+2s+m$. 
\end{theorem}

The proof of this theorem follows directly from \cref{explicit:arc:system:maximal} (lower bound) and \cref{lem:dimension:maximal:arc:system} (upper bound).

\begin{lemma}\label{lem:arc:system:triangle}
    Let $\asA$ an arc system in $F_0$ such that all the pieces of the splitting of $F_0$ along $\asA$ are triangles except for one of the pieces which is a triangle $T$ with $m$ marked points and $n$ punctures within its interior. Then there is an arc system $\asB$, obtained by adding $3m+2n$ $\Delta$-arcs to $\asA$, that splits $F_0$ into triangles and once-punctured monogons.
\end{lemma}
\begin{proof}
    For $m=1$ and $n=0$, $\asB$ is formed by adding three $\Delta$-arcs, splitting $T$ into three triangles.
    For $m=0$ and $n=1$, $\asB$ is constructed by adding two $\Delta$-arcs and the pieces of the splitting of $T$ along $\asB$ are two triangles and a one-punctured monogon. 
    In the general case, $\asB$ is formed by adding three $\Delta$-arcs for each marked point and two $\Delta$-arcs for each puncture, resulting in $3m+2n$ $\Delta$-arcs, see Figure \ref{fig:triangle:arc:system}.
    %%%%%%%%%%%%%%%%%%%%%
    %%%%%% FIGURE %%%%%%%
    %%%%%%%%%%%%%%%%%%%%%
    \begin{figure}[h]
        \centering
        \includegraphics[width=0.5\textwidth]{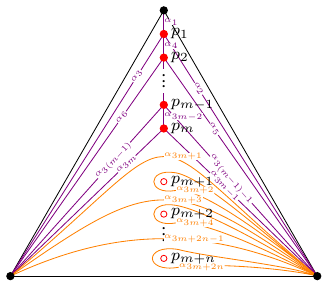}
        \caption{\small Maximal arc system in a triangle.}
        \label{fig:triangle:arc:system}
    \end{figure}
    %%%%%%%%%%%%%%%%%%%%%%%%
    %%%%%%%%%%%%%%%%%%%%%%%%
\end{proof}

\begin{lemma}\label{explicit:arc:system:maximal}
    There exists an arc system $\asA$ in $F_0$ of rank $6g-7+2s+m$ that splits $F_0$ into triangles and once-punctured monogons.
%    Let $F_0$ be a surface of genus $g$ with $m$ marked points and $n-m$ punctures. Then there exists an arc system $\asA$ of $F_0$ of rank $6g-7+2n+m$ that split $F_0$ into triangles and once-punctured monogons.
\end{lemma}
\begin{proof}
    For simplicity, let $n=s-m$ denote the number of punctures of $F_0$. We will show that there is an arc system which has $6g-6+3m+2n$ $\Delta$-arcs that splits $F_0$ into triangles and once-punctured monogons. The proof is divided into cases. 
        %We denote by $S_{\asA}$ the splitting of $F_0$ along the arc system $\asA$.
        
    \textbf{Case 1: $g=0$.}

    For $m=1$ and $n=2$, $\asA$ consists of a single $\Delta$-arc surrounding one puncture, splitting $F_0$ into two once-punctured monogons. Note that $6(0)-6+3(1)+2 = 1$, as expected. 
    For $m=1$ and $n\geqslant 3$, we take three $\Delta$-arcs based on the marked point and each of these surrounding a single puncture. By splitting $F_0$ along these three $\Delta$-arcs, we obtain a triangle with $n-3$ punctures and $3$ once-punctured monogons. By \cref{lem:arc:system:triangle}, we add $2(n-3)$ $\Delta$-arcs in the triangle with $n-3$ punctures, resulting in the desired arc system $\asA$, with $3+2(n-3) = 6(0)-6+3(1)+2n$ $\Delta$-arcs.
    Similar procedures apply to cases where $m\geqslant 2$, by selecting $\Delta$-arcs such that one of the pieces resulting from the splitting $F_0$ along these $\Delta$-arcs is a triangle which contain the remaining points of $\Delta$ and the punctures. Then by \cref{lem:arc:system:triangle}, we add the remaining $\Delta$-arcs. 

    \textbf{Case 2: $g\geqslant 1$.}

    Consider a $\Delta$-arc system with $2g$ arcs, all based in a single point of $\Delta$, such that the system splits $F$ in a $4g$-polygon.  Consider first the arc system $\asA'$ %Th arc system$$\asA$  is 
    constructed using $2g$ $\Delta$-arcs on the sides of the polygon, and  $4g-3$ $\Delta$-arcs that connect a vertex of the polygon with the remaining ones. 
    These $6g-3$ $\Delta$-arcs are selected such that the splitting of $F_0$ along them, only one piece contains the $n$ punctures and the remaining $m-1$ points of $\Delta$ in its interior (see Figure \ref{fig:arc:system:A}). 
    Finally, by \cref{lem:arc:system:triangle} we add $3(m-1)+2n$ $\Delta$-arcs to the arc system $\asA'$ in the special triangle, obtaining the desired arc system $\asA$, with $6g-3+3(m-1)+2n = 6g-6+3m+2n$ $\Delta$-arcs (see Figure \ref{fig:arc:system:B}). 
    
%%%%%%%%%%%%%%%%%
%%% FIGURE %%%%%%
%%%%%%%%%%%%%%%%%
\begin{figure}[h!]
    \centering
    % Primera subfigura
    \begin{subfigure}[b]{0.45\textwidth}
        \centering
        \includegraphics[width=\textwidth]{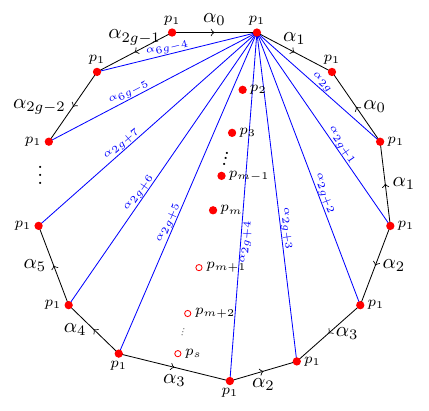}
        \caption{}
       \label{fig:arc:system:A}
    \end{subfigure}
    \hfill
    % Segunda subfigura
    \begin{subfigure}[b]{0.45\textwidth}
        \centering
        \includegraphics[width=\textwidth]{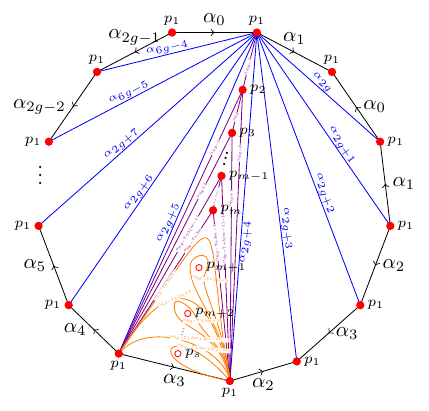}
        \caption{}
        \label{fig:arc:system:B}
    \end{subfigure}
    \caption{\small Arc systems  $\asA'$ and $\asA$ in a surface $F_0$}
    \label{fig:arc:system}
\end{figure}
%%%%%%%%%%%%%%%%%
%%%%%%%%%%%%%%%%%
\end{proof}
%%%%%%%%%%%%%%%%%
%%%%%%%%%%%%%%%%%

%%%%%%%%%%%%%%%%%
%%%% LEMMA %%%%%%
%%%%%%%%%%%%%%%%%
\begin{proposition}[Characterization of maximal arc systems]\label{prop:characterization:maximal:arc:systems}
    Let $\asA$ be an arc system in $F_0$. The arc system $ \asA$ is maximal if and only if the pieces of the splitting of $F_0$ along $\asA$ are triangles or once-punctured monogons. 
\end{proposition}
\begin{proof}
    Suppose that $\asA$ is maximal, and let $B$ be a piece of the splitting of $F_0$ along $\asA$. We have that $B\cong F_{h,b}^{n}$, where $h$ is the genus of the surface, $b$ is the number of boundary components, and $n$ is the number of punctures. Let $\ell$ be the number of marked points on the boundary of $B$. 
    Note that $b, \ell \geqslant 1$, and there is at least one marked point in each boundary component. We will prove that $h=0, b=1$ (a disk), and either $n=1$ and $\ell=1$ (a once-punctured monogon), or $n=0$  and $\ell=3$ (a triangle). 
    If $h\geqslant 1$, $b \geqslant 2$, $n\geqslant 2$, or $\ell\geqslant 4$, we can add $\Delta$-arcs as shown in \cref{fig:EXTa} to extend the arc system $\asA$ into a new one, which contradicts the maximality of $\asA$.
    Thus, $h=0, b=1, n \leqslant 1$, and $\ell \leqslant 3$. We exclude the case $n=1$ and $\ell=2,3$ since we can add new $\Delta$-arcs to $\asA$ and obtain a new arc system, contradicting the maximality of $\asA$ (see \cref{fig:EXTb}). Additionally, we exclude when $n=0$ and $\ell=1,2$, since the boundary of the piece $B$ represents $\Delta$-arcs of $\asA$, and in these cases, the piece $B$ gives us two arcs that are isotopic relative to $\Delta$ (when $\ell=2$) or a $\Delta$-arc which is isotopic to the marked point (when $\ell=1$). Therefore $B$ is either a triangle ($h=0, b=1, n=0$, and $\ell=3$) or a once-punctured monogon ($h=0, b=1, n=1$, and $\ell=1$).
    %%%%%%%%%%%%%%%%%
    %%%% FIGURA %%%%%
    %%%%%%%%%%%%%%%%%
    \begin{figure}[h]
    \centering
    \begin{subfigure}{1\textwidth}
        \centering
        \includegraphics[width=0.23\linewidth]{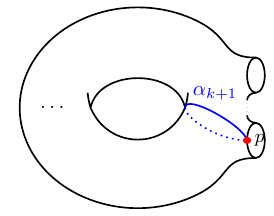}
            %\hspace{0.03\textwidth}
        \includegraphics[width=0.23\linewidth]{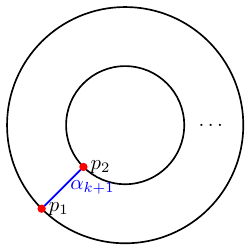}
            %\hspace{0.03\textwidth}
        \includegraphics[width=0.23\linewidth]{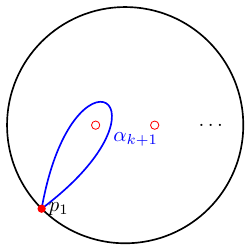}
            %\hspace{0.03\textwidth}
        \includegraphics[width=0.23\linewidth]{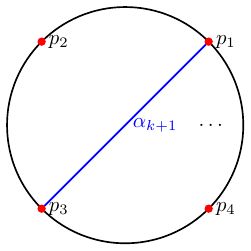}
        \caption{}
        \label{fig:EXTa}
    \end{subfigure}
    \hspace{1cm} % Espacio entre subfiguras
    \begin{subfigure}{1\textwidth}
        \centering        \includegraphics[width=0.23\linewidth]{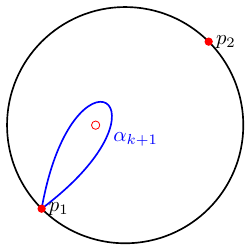}
        \hspace{0.1\textwidth}
        \includegraphics[width=0.23\linewidth]{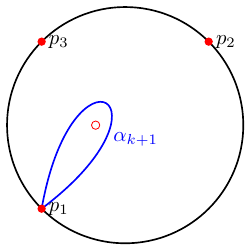}
        \caption{}
        \label{fig:EXTb}
    \end{subfigure}

    \caption{\small $\Delta$-arcs for extending the arc system}
 \label{fig:extending_arc_systems}
    %\label{MAxMIN}
\end{figure}
    %%%%%%%%%%%%%%%%%
    %%%%%%%%%%%%%%%%%

    Now, consider an arc system $\asA$ that splits $F_0$ into triangles or once-punctured monogons. Suppose $\asA$ is not maximal. Thus, there exists an arc system $\asA'$ such that $\asA \subsetneq \asA'$. Let $\alpha \in \asA' - \asA$. Then, $\alpha$ lies in a piece $B$ of the splitting of $F_0$ along $\asA$, which is either a triangle or a once-punctured monogon.
    Then, $\alpha$ is isotopic to a $\Delta$-arc of $\asA$ or a marked point. However, this contradicts the assumption that $\asA'$ forms an arc system. Therefore, $\asA$ is maximal.
\end{proof}

% \begin{lemma}\label{max:sys:cut:trian}
% Let $A$ be a maximal arc system, then any piece of $F_0-A$ is either a triangle or a once-punctured monogon.
% \end{lemma}
% \begin{proof}
%  \textcolor{blue}{Nestor}   
% \end{proof}

\begin{lemma}\label{lem:dimension:maximal:arc:system}
Let $\asA$ be a maximal arc system in $F_0$. Then the $\rank$ of $\asA$ is $6g-7+2s+m$.
\end{lemma}
\begin{proof}

 Let $\asA$ be a maximal arc system in $F_0$. We have two cases $m=s$ or $m\neq s$.
In the first case, we have by \cref{prop:characterization:maximal:arc:systems} and \cref{arc:system:deter:CW:structure} that $\asA$ induces a CW-structure in $F$. Denote by $V$ the number of vertices, $E$ the number of edges, and $T$ the number of triangles.  Then 
$\chi(F)=V-E+T=m-E+T$.
Note that each arc labels exactly two edges of the triangles, hence we have  $E=\frac{3T}{2}$. Thus $\chi(F)=2-2g=m-E+\frac{2E}{3}$,
it follows that $E=6g-6+3m$. Therefore $\asA$ has $\rank$ $6g-7+3m$. 

Now suppose that $s\neq m$.  Note that any two once-punctured monogons in the splitting of $F_0$ along $\asA$  cannot share a side unless the surface is homeomorphic to the sphere with two punctures and one marked point. If $F_0$ is the sphere with two punctures and one marked point the claim is clear, then assume that we are not in this case. Therefore any once-punctured monogon shares a side with a triangle.

Note that in the splitting of $F_0$, there are exactly $s-m$ pieces that are once-punctured monogons. Additionally the closed surface resulting from $F_0$ by capping the punctures is $F$. Hence, by the previous observation, when we cap the punctures in $F_0$  we need to remove exactly 
$2(s-m)$ arcs to obtain an arc systems $\asA'$ in $F$. By the closed case we have $\rank(\asA')=6g-7+3m$. It follows that $\rank(\asA)=6g-7+2s+m$. 
\end{proof}

\begin{proposition}\label{prop:A:contained:in:maximal}
Any arc system $\asA=\langle \alpha_0 ,\ldots , \alpha_k \rangle$ in $F_0$ is contained in a maximal arc system. In particular, an arc system of rank $6g-7+2s+m$ is maximal.
\end{proposition}
\begin{proof}
Suppose that $\asA$ is not maximal. By Proposition \ref{prop:characterization:maximal:arc:systems}, the splitting of $F_0$ along the arc system $\asA$ yields at least one piece that is neither a triangle nor a once-punctured monogon.
As in the proof of \cref{prop:characterization:maximal:arc:systems}, we can add a $\Delta$-arc $\alpha_{k+1}$ as described in \cref{fig:extending_arc_systems} extending the arc system $\asA$ into a new system $\asA'$. By \cref{lem:dimension:maximal:arc:system} we add a finite number of $\Delta$-arcs to obtain a maximal arc system which contains $\asA$. 
%\comn{Mencionar que por el lema anterior este proceso es finito.}
\end{proof}

\section{Harer's spine and its dimension}\label{sec:spine}

In this section we recall the definition of Harer's spine $\mathcal{Y}$ from \cite[Section 2]{Ha86} and compute its dimension proving \cref{thm:main}. Let \(1\leq  m \leq s \),   \( \Delta = \{p_1, \ldots, p_m\} \), \( P = \{p_{m+1}, \ldots, p_s\} \) and \( F_0 = F - P \) as in \cref{sec:arcs}. 

The {\it decorated Teichm\"uller space} $\Tgs(\Delta)$ is the space of all pairs $(R,\lambda)$ where $R$ is a point of $\Tgs$ and $\lambda$ is a projective class of a collection of positive weights on the $m$ points of $\Delta$. Topologically $\Tgs(\Delta)$ is homeomorphic to the product of $\Tgs$ and a open simplex of dimension $m-1$. There is a natural diagonal action of $\pmodgs$ on $\Tgs(\Delta)$ when $m\leq s$, and only when $m=s$ it actually admits an action of $\modgs$.%only admits an action of $\modgs$ provided $m=s$.

 We say an arc system $\asA$ \emph{fills up} $F_0$  provided every point in $\Delta$ is the initial or final point of an arc in $\asA$, and each connected  component $B$ of the splitting of $F_0$ along $\asA$ is either homeomorphic to a disk or a once-puntured disk.  As a direct consequence of \cref{prop:characterization:maximal:arc:systems} every maximal arc system fills up $F_0$. Note that when $P$ is empty, that is we have no punctures, $\asA$ determines a CW-structure on $F$ such that $\Delta$ is the 0-skeleton and the union of the arcs in $\asA$ form the 1-skeleton.

  Let $\calA_\infty(\Delta)$ be the codimension $2$ subcomplex of $\calA(\Delta)$ formed by all the arc systems that do not fill up $F_0$.  From the definition we can see that  $\PMod_g^s$ acts simplicially on $\calA(\Delta)$ and preserves the subcomplex $\calA_\infty(\Delta)$. Notice that  it admits an action of $\Mod_g^s$ only when $s=m$. The following result  states that $\mathcal{A}(\Delta)-\mathcal{A}_{\infty}(\Delta)$ defines an ideal triangulation of the decorated Teichm\"uller space $\Tgs(\Delta)$.

\begin{theorem}\cite[Theorem 1.3]{Ha86}\label{Theo:HomeoTeich} There is a $\pmodgs$-equivariant homeomorphism $$\phi:\Tgs(\Delta)\longrightarrow (\mathcal{A}(\Delta)-\mathcal{A}_{\infty}(\Delta)).$$ 
\end{theorem}

Harer's proof of this theorem in \cite{Ha86} is in the conformal category and is due to Mumford. In \cite[Chapter 2]{HarerSurvey} Harer attributes to W. Thurston the original construction of this ideal triangulation. A proof using hyperbolic geometry was provided by B.H. Bowditch and D.B. A. Epstein \cite{BE88}. See also the work of R. C. Penner \cite{Penner87} for an alternative proof. In the statement of \cref{Theo:HomeoTeich} we can replace $\pmodgs$ by $\modgs$ when $s=m$.
  
Let us denote by $\calA^0$ and by $\calA_{\infty}^0$ the barycentric subdivisions of the complexes $\calA(\Delta)$  and $\calA_\infty(\Delta)$, respectively. 

\begin{definition}[Harer's spine] Let $\calY(=\calY_g^{s,m})$ be the subcomplex of $\calA^0$ spanned by the arc systems that fill up $F_0$. Therefore, the vertices of $\calY$ consist of the set of arc systems that fill up $F_0$ and an $\ell$-simplex of $\calY$ is given by a chain of arc systems
\[\asA_0 \subset \cdots \subset \asA_\ell,\]
where all the inclusions are strict and all the $\asA_i$ fill up $F_0$.  We call {\it Harer's spine} both the complex $\calY$ and the subspace $\phi^{-1}(\calY)\subset \Tgs(\Delta)$.
\end{definition}
    
\begin{remark} We are mostly following the notation from \cite{Ha86}. However, instead of using $\mathcal{Y}$,  Harer uses the notation $Y^0$ for the spine, since (as he explains) it coincides with the first barycentric subdivision of a complex $Y$ given by the  dual of $\mathcal{A}$. 
\end{remark}

 Harer states the following result: 

\begin{theorem}\cite[Theorem 2.1]{Ha86}\label{Theo:DimSpine}
The complex $\mathcal{A}^0$ equivariantly deformation retracts onto the $4g - 5 + s + m$ dimensional complex $\mathcal{Y}$. This retraction provides a $\PMod_{g}^s$-equivariant homotopy equivalence between $\mathcal{A}^0 - \mathcal{A}_\infty^0$ and $\mathcal{Y}$. 
\end{theorem}

 Harer's computation of the dimension of $\calY$ is correct when $m=s$, however is off by $1$ when $m<s$ as the following example and our proof of \cref{thm:main} below show.

\begin{example}[Case $g\geq 1$, $s=2$, 
 $m=1$] \label{EXA} Let $F$ be a closed surface of genus $g\geq 1$ with $s=2$ marked points $\{p_1,p_2\}$ and we take $\Delta=\{p_1\}$ and $P=\{p_2\}$. From  \cite[Theorem 2.1]{Ha86} the dimension of Harer's spine $\mathcal{Y}$ is equal to \(4g -2=\vcd(\Mod_{g}^2)\).  However,  let us exhibit a $(4g-1)$-simplex of $\mathcal{Y}$.

 Consider a CW-structure of the surface $F$ given by the $0$-cell ${p_1}$ and $1$-cells $\alpha_0,\ldots,\alpha_{2g-1}$ that cut the surface into a $4g$-gon; see figure \ref{MAxMINa}. Notice that the arc system \(\asA_0=\langle \alpha_0,\ldots,\alpha_{2g-1}\rangle\) cuts the surface $F_0=F-\{p_2\}$ into a $4g$-gon with a puncture in $P$, hence it fills up $F_0$ and it gives a vertex of $\mathcal{Y}$ of minimal weight$=2g-1$.    By adding the $4g-3$ orange arcs $\alpha_{2g},\ldots \alpha_{6g-4}$ and the $2$ violet arcs $\alpha_{6g-3}$ and $\alpha_{6g-2}$ from figure \ref{MAxMINb}, we obtain an  arc system $\asA_{4g-1}$ of rank $6g-2$ which fills up  the surface $F_0$ and gives a vertex of  $\mathcal{Y}$ of maximal weight$=6g-2$.  Hence we obtain the chain 
 $$ \asA_0 \subset \asA_1  \subset \cdots \subset \asA_{4g-1},$$
 by adding an arc at a time, which defines a $(4g-1)$-simplex of $\mathcal{Y}$.

 %This is because the generators of the fundamental group of \( S_g \) form a "minimal vertex" in \( Y^0 \) of rank \( 2g - 1 \). Then this system can be extended to a maximal one of rank \( 6g - 2 \), see \cref{ex1}.

\begin{figure}[h]
    \centering
    \begin{subfigure}{0.4\textwidth}
        \centering
        \includegraphics[width=\linewidth]{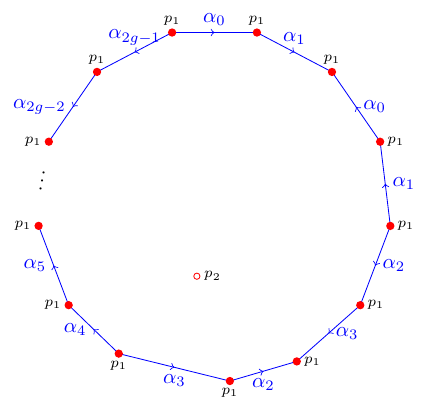}
        \caption{}
        \label{MAxMINa}
    \end{subfigure}
    \hspace{1cm} % Espacio entre subfiguras
    \begin{subfigure}{0.4\textwidth}
        \centering        \includegraphics[width=\linewidth]{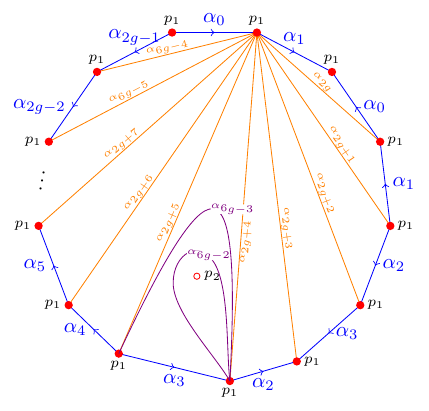}
        \caption{}
        \label{MAxMINb}
    \end{subfigure}

    \caption{\small Arcs that give the arc systems $\asA_0$ of minimal  weight  and $\asA_{4g-1}$ of maximal weight.}
 
    %\label{MAxMIN}
\end{figure}

\end{example}

%\begin{theorem}[Dimension of Harer's spine]\label{thm:dim:Harer:spine}
%Let $s\geq 1$, \  $2g+s>2$, and $1\leq m\leq s$. The dimension of $\calY$ equals 

 %   \[
%5\dim(\calY) = \begin{cases}
 %   4g-4+s+m & \text{ if } m<s\\
  %  4g-5+s+m & \text{ if } m=s.
%\end{cases} 
%\]
%\end{theorem}
\begin{proof}[Proof of \cref{thm:main}]
    Note that, whenever $\asA_1\subset \asA_2$ and $\asA_1$ fills up $F_0$, we have that $\asA_2$ also fills up.  Recall that every maximal arc system in $F_0$ fills up the surface and is of rank $6g-7+2s+m$ (\cref{lem:dimension:maximal:arc:system}), and every arc system is contained in a maximal one (\cref{prop:A:contained:in:maximal}). Hence to compute the dimension of $\calY$ it is enough to determine the rank of a minimal arc system that fills up $F_0$. Now we split the proof into two cases.

    \textbf{Case 1: $m<s$.} In this case the set $P$ is non-empty and $F_0$ is a punctured surface. 
    
    We claim an arc system $\asA$  is minimal among the arc systems that fill up $F_0$ if and only if it splits $F_0$ into pieces all of which are once-punctured discs.  First assume  that $\asA$ fills up $F_0$, it is minimal and there is a connected component $B$ of corresponding splitting that is homeomorphic to a disc. Then, by connectedness of $F_0$, an edge of $B$ must have the same label as an edge of another connected component (this connected component exists because $P$ is nonempty). Thus removing the corresponding arc we get a strictly smaller arc system that fills up $F_0$. For the converse, assume that $\asA$ splits $F_0$ into once-punctured discs and let $B$ be a connected component of the corresponding splitting. Then removing one arc from $B$ will have the effect of either gluing two different pieces or gluing two edges of the same piece, in either case the resulting arc system do not fill up $F_0$.

 Let $\asA$ a minimal arc system in $\calY$, we claim the rank $r$ of $\asA$ is $2g+s-3$. Then $\asA$ induces a CW-structure on $F$ with $m$ 0-cells, $r+1$ 1-cells, and $s-m$ 2-cells (one for each once-punctured disc in the splitting of $F_0$). Hence, as $F$ is a closed surface of genus $g$, the Euler characteristic of $F$ leads to 
 \[2-2g=m-(r+1)+(s-m)=-r+s-1,\]
hence $r=2g+s-3$. Therefore
\[\dim(\calY)=(6g-7+2s+m)-(2g+s-3)=4g-4+s+m.\]

\textbf{Case 2: $m=s$.} In this case the set of punctures is empty, and any arc system that fills up the surface determines a cellular structure on $F=F_0$. We claim an arc system is minimal among the arc systems that fill up if and only if it splits $F$ into exactly one disc. The proof of this claim is similar to the analogous claim in the previous case. A similar Euler characteristic computation as in the previous case implies that the rank $r$ of a minimal arc system that fills up $F$ satisfies
\[2-2g=s -(r+1)+1=s-r,\]
hence $r=2g+s-2$. Therefore
\[\dim(\calY)=(6g-7+2s+m)-(2g+s-2)=4g-5+s+m.\]
    
\end{proof}

We end by explaining how the dimension of  Harer's spine when $s=1$, and a Briman exact sequence argument, imply the existence of a model for $\underline E \pmodgs$ of minimal dimension for $s\geq 2$. In \cite[Theorem 4.2]{Gui2010}, G. Mislin also used a Birman exact sequence argument to show the existence of co-compact models for $\underline E \pmodgs$; see also the argument in \cite[Remark 4.5]{AMP14} which applies for pure mapping class groups.

Recall that the \emph{proper geometric dimension of a group $G$}, denoted $\gdfin(G)$, is the minimum $n$  for which there exists an $n$-dimensional model for $\underline{E}G$.  If $G$ is virtually torsion free, then $\vcd(G)\leq \gdfin{G}$ (\cite[Theorem 5.23]{Lu05}).

\begin{corollary}\label{Propgd}
For $s\geq 1$   and  $2g+s>2$, there exists a model for $\underline E\pmodgs$ of dimension equal to $\vcd(\pmodgs)$.
\end{corollary}
\begin{proof}
%We know that $\vcd(\pmodgs)\leq\gdfin(\pmodgs)$.
 %On the other hand, 
 From \cref{thm:main},  for $g=0$ and $s\geq 3$, the dimension of Harer's spine  $\mathcal{Y}$ is $\vcd(\PMod_0^s)=s-3$, and for $g\geq 1$ and $s=1$, we have that $\dim \mathcal{Y}=\vcd(\PMod_g^1)=4g-3$. Therefore, in these cases Harer's spine $\mathcal{Y}$ is a model of $\underline{E}\PMod_g^s$ of minimal dimension.

Now let $s\geq 2$  and $g\geq 1$. Then $\pi_1(F_g^{s-1})$ is a nonabelian free group and recall that, due to a well-known theorem by Dunwoody, any finite extension $\Delta$ of a non trivial finitely generated free group  has $\gdfin(\Delta)=1$. Hence, from \cite[Theorem 5.16]{Lu05} applied to the Birman short exact sequence 
$$1\rightarrow\pi_1(F_g^{s-1})\rightarrow \pmodgs\rightarrow \PMod_g^{s-1}\rightarrow 1$$
it follows from an inductive argument that
$$\gdfin(\pmodgs)\leq \gdfin(\PMod_g^{s-1})+ 1\leq \gdfin(\PMod_g^{1})+ (s-1)=\vcd(\PMod_g^1)+(s-1) =\vcd(\pmodgs).$$
\end{proof}

%%%%%%%%%%%%%%%%%%%%%%%%%%%%%%%%%%%%%%%%%%%%%%%%%
\bibliographystyle{alpha} %harvard, unsrt, alpha
\bibliography{mybib}
\end{document}